\documentclass[12pt]{amsart}
\usepackage{amsthm,amsmath,amssymb,mathrsfs,amsbsy,bbm}
\usepackage{comment}
\usepackage{fullpage}
\usepackage[final]{pdfpages}
\usepackage[utf8]{inputenc}
\usepackage[english]{babel}
\usepackage[colorlinks,linkcolor = red]{hyperref}
\usepackage[noabbrev,capitalize]{cleveref}

\hypersetup{ linktoc=page }
\newtheorem{theorem}{Theorem}[section]
\newtheorem*{main theorem}{Main Theorem}
\newtheorem{definition}[theorem]{Definition}
\newtheorem{claim}[theorem]{Claim}
\newtheorem{conjecture}[theorem]{Conjecture}
\newtheorem{lemma}[theorem]{Lemma}
\newtheorem{example}[theorem]{Example}

\newtheorem{remark}[theorem]{Remark}

\newcommand{\Gl}[3]{\mathrm{GL}_{#1}(\mathbb{#2}_{#3})}
\newcommand{\bb}[1]{\mathbb{#1}}
\newcommand{\rmf}[1]{\mathrm{#1}}

\newcommand{\InPr}[2]{\left < {#1}, {#2} \right >}

\newtheorem{introtheorem}{Theorem}

\catcode`\@=11
\newdimen\cdsep
\def\cdstrut{\vrule height .6\cdsep width 0pt depth .4\cdsep}
\def\@cdstrut{{\advance\cdsep by 2em\cdstrut}}

\def\arrow#1#2{
	\ifx d#1
	\llap{$\scriptstyle#2$}\left\downarrow\cdstrut\right.\@cdstrut\fi
	\ifx u#1
	\llap{$\scriptstyle#2$}\left\uparrow\cdstrut\right.\@cdstrut\fi
	\ifx r#1
	\mathop{\hbox to \cdsep{\rightarrowfill}}\limits^{#2}\fi
	\ifx l#1
	\mathop{\hbox to \cdsep{\leftarrowfill}}\limits^{#2}\fi
}
\catcode`\@=12

\title{Distinguished Representations with respect to Symmetric Subgroups of $GL_{n}(\bb{F}_{q})$}
\author{Guy Kapon}
\date{\today}

\begin{document}

\maketitle
\begin{abstract}
    We study representations of $GL_{n}(\bb{F}_{q})$ that are distinguished with respect to a symmetric subgroup $H=GL_{n}(\bb{F}_{q})^{\sigma}$, where $\sigma$ is an involution. We prove that those representations satisfy $\pi \cong \pi ^{*,\sigma}$, thus positively answering a version of the Prasad-Lapid conjecture.
\end{abstract}
\tableofcontents
\section{Introduction}

%Things I at least at one point considered should be in the intro===
%RELATIVE REP THEORY

%EREZ CONJECTURE

%KNOWN THINGS

%WHAT WE PROVE

%check jaques rallis spelling everywhere ===
%I did
%Also fixed jaques to Jaquet without rallis.
%Should probably do again.

\subsection{Main Results}
Relative representation theory is concerned with doing harmonic analysis of spaces of functions $C(X)$ with a group $G$ acting on $X$. Of particular interest to many is the case where $G$ is an algebraic group (usually over a local field) and $X=G/H$ is a symmetric space, i.e., there is an involution $\sigma$ of $G$, and $H=G^{\sigma}$.

In \cite{prasad2015relative} Prasad made a conjecture which was later generalized by Erez Lapid for a necessary condition for an irreducible representation of $G$ to be $H$-distinguished, which means that it appears in $C(G/H)$, in the case of symmetric pairs.

%I also have tau here###
%It might be fine since this is not in my notions anyway.
\begin{conjecture}
(Lapid-Prasad) Let $G$ be a connected reductive algebraic group defined over a local non-archimedean field $F$, $\theta:G\rightarrow G$ an involution defined over $F$ and $\rmf{H}=G^{\theta}$. Let $\pi $ be a smooth irreducible representation of $G(F)$. If $\pi$ is $\rmf{H}-$distinguished then the L-packet\footnote{We will not define L-packets as we do not need them here, for an expository note see \cite{article}} of $\pi$ is invariant to the functor $\tau \rightarrow \tilde{\tau}^{\theta}$, where $\tilde{\tau}$ is the contragredient representation and $\tau^{\theta}$ is the twist by $\theta$.
\end{conjecture}

%???===
%Fixed the sentence (which is I belive what I meant).
There is also an analogous conjecture for the Archimedean case.

%I dont know if this is needed===###
%defintly can be, I can decide later. 

Several cases of this were proved: 
\begin{itemize}
% \cite{uniq} ===
%the citing here didn't work for some reason
%apperntly i didn't copy it from the bib of the previous paper.

\item For the pair $(\Gl{a+b}{Q}{p},\Gl{a}{Q}{p} \times \Gl{b}{Q}{p})$ this was proved in \cite{uniq}.

\item The case where $G =\mathrm{GL}_{n}(E)$ for quadratic extensions $E/F$, and $\rmf{H}$ is a subgroup of unitary transformations was proved in \cite{Feigon2012OnRD}.

\item The case of real Galois symmetric pairs was proved in \cite{2017}

%captalize?==
%Im not sure what I meant, maybe that i need to cap' in general? Im adding a general things to do list above.

\end{itemize}
Another natural setting where this conjecture can be stated is over finite fields.  In \cite{DisGK} this was proved for the Jacquet-Rallis case, in this paper we generalize to all symmetric subgroups of $GL_{n}(\bb{F}_{q})$.

Throughout this paper $q$ is an odd prime power.

%Rami said i should only have theorem B===
%Still need to do that, look at rami stuff to make sure how he handled it?

%Need to add some words here.

Our main result is:
%non numbered  CALL THEOREM A AND B===
%did This longgggg ago.

\begin{introtheorem}\label{theorem A}
    Let $G=Gl_{n}(\bb{F}_{q^{m}})$ where $m$ is either $1$ or $2$, let $\sigma$ be an algebraic involution, defined over $\bb{F}_{q}$, let $H=G^{\sigma}$ be a symmetric subgroup, and let $\pi \in Irr(G)$ be an irreducible representation. If $\pi^{H}\neq 0$ then $\pi \cong \pi^{*}\circ \sigma$.
\end{introtheorem}

%DISCUSSION OF DIFFRENT INVOULTIONS?===
%I wrote it, its in a diffrent place, I might want a remark  here but I dont think its important.
%OH, I wanted to have a section here? I feel it is unnnecesey, but maybe ill ask rami.

%\subsection{Distinguished Representations and Symmetric Pairs}
%This is the possible subsection.

\subsection{Ideas of The Proof}

\subsubsection*{Cuspidal Case}

This case is done using Deligne-Lusztig induction, in fact, most of this case is included in a theorem by Lusztig:

\begin{theorem}[{{\cite[Theorem 3.5]{Lusztig2007}}}]\label{luzTh}
    Let $\rho$ be an irreducible representation of $G(\bb{F}_{q})$ which is $H-$distinguished. suppose that  $\rho$ appear in $H^{i}_{c}(X)^{\theta}$ (Where $X$ is the Lusztig variety corresponding to some torus $T$). Then there is some $f\in G(\bb{F}_{q^{n}})$ such that $T^{f}$ is $\sigma$ invariant, and $\tilde{\theta}=\theta(N(f^{-1}tf))$ as a character on $T^{f}(\bb{F}_{q^n})$ satisfies $\tilde{\theta}(\sigma(t)) = \tilde{\theta} (t)^{-1}$.
\end{theorem}

In simpler words, it says that if $R_{T,\theta}$ has a distinguished representation (or even if only in one of the cohomologies appearing in its definition), then $\theta^{-1}$ and $\theta \circ \sigma$ are geometrically conjugate. 

In particular, this proves:
\begin{theorem} \label{LT}
    If $W_{\theta}$ is trivial, and $R_{T,\theta}$ is $H-$distinguished, then $R_{T,\theta} \cong R_{T,\theta} ^{*,\sigma}$.
\end{theorem}

%maybe formal formulation, maybe refrence===
%No need for the formal formulation here (it is later) if i find a refrence then sure.
%The wording here howeevr can be bettered.

%did words, Should still maybe fix them a bit, a refrence would also be favorable.
Together with the fact that any cuspidal representation of $Gl_{n}$ is of the form $R_{T,\theta}$ (up to sign) for $T$ an anisotropic torus, and $\theta$ in general position, this finishes the proof in the cuspidal case.

This fact is standard when $Gl_{n}(\bb{F}_{q})$ is considered over $\bb{F}_{q}$, and for $Gl_{n}(\bb{F}_{q^{2}})$ as a group over a $\bb{F}_{q}$ we show that the interaction of Deligne-Lusztig induction with restriction of scalars is as nice as one can expect:

\begin{theorem}\label{IR}
    The set of $T,\theta$ a maximal torus and a character on its rational points, is in bijection for  $G(\bb{F}_{q^n})$ as a group over $\bb{F}_{q^n}$, and for the restriction of scalars of $G$ to $\bb{F}_{q}$ as a group over $\bb{F}_{q}$. $R_{T,\theta}$ for corresponding elements is the same up to sign, whether we regard $G$ as a group over $\bb{F}_{q^n}$ or do restriction of scalars and consider it as a group over $\bb{F}_{q}$.
\end{theorem}
This extends the previous results to this case.

\subsubsection*{The General Case} 

This Case follows closely what was done in \cite{DisGK}, again using the Zelevinsky classification of representations of $GL_{n}$.

First, we prove:
%something has to be preserved by sigma?==
%I dont think it needs, but I should make sure about this sort of stuff.
\begin{lemma}\label{PindLem}
    Let $G,\sigma,H$ be as in \cref{theorem A}. let $L\subseteq P$ where $L$ is a Levi subgroup and $P$ is a Parabolic of $G$ , $\pi$ a representation of $L$. If $i_{L}^{G}(\pi)$ is $H$-distinguished, then there is a levi subgroup $L'\subseteq L$ and an element $A\in G$ such that the map $X\rightarrow A\sigma (X)A^{-1}$ is an involution which preserves $L'$, and $r_{L'}^{L}(\pi)$ is distinguished with respect to this involution.
\end{lemma}
In addition, the proof can be seen to describe $L'$, namely, it will be the maximal Levi such that the involution acts on it.

The proof of the lemma starts by using Mackey theory to get a $P \cap H^g$-distinguished representation, for some $g$. Then, using some information on $P\backslash G/H$, we find a unipotent subgroup that also acts trivially, getting a distinguished Jaquet restriction.

%sketch of proof?===
%Added One.

The proof then proceeds in two steps. First, there is the $\rho$-isotopic case. If $\rho$ is a cuspidal representation, then a representation is called $\rho$-isotopic if it is a subrepresentation of $i(\rho^{n})$.

If a $\rho$-isotopic representation is distinguished then applying $\cref{PindLem}$ yields:

\begin{theorem}\label{IsoCase}
    if $\rho \in Irr(GL_{n})$ is a cuspidal representation, and $x\in Irr(GL_{nk})$ is a $\rho$-isotopic distinguished representation, then $\rho$ is invariant under the corresponding type of involution. 
\end{theorem}
%EXPLAIN===
%it's the invoultion type thing coming to bite me again.
%also bellow same thing.
%Should also make the text better.
%made texts somewhat better, can be improved?
Here, involution type refers to an involution up to conjugation. In $GL_{n}$ there are four involution types, and those four are consistent as $n$ changes. While being distinguished depends on the involution, being isomorphic to the dual up to a $\sigma$ twist only depends on the type.

A sketch of the proof is as follows. Once we use \cref{PindLem}, we see we either have a block preserved by the involution given in the lemma, in which case the cuspidal case of \cref{theorem A} applies, or we have two blocks that are interchanged, in which case it follows directly.

The isotopic case is then completed using a claim from Zelevinsky theory:

\begin{theorem}\label{zelThm}
    If $\rho$ is a cuspidal representation preserved by an involution type, then so is any $\rho$-isotopic representation.
\end{theorem}

For the final general case, Zelevinsky theory tells us that any irreducible representation is the parabolic induction of a tensor of $\rho_{i}$-isotypic representations for different $\rho_{i}$. In this case, the Jaquet restriction in $\cref{PindLem}$ does not immediately have to be trivial. 

If any block in $L'$ is fixed by the involution, then since the Jaquet restriction of a $\rho$-isotypic representation is still $\rho$-isotypic, this would mean $\rho$ is involution type invariant.
If there are two blocks that are interchanged, then we get directly that if they are in the blocks of $L$ corresponding to $\rho,\rho'$, then $\rho^{\sigma,*}=\rho'$, This means that for each $\rho$, either it is invariant, and then its block is preserved, or it comes with a paired $\rho'$ and their blocks are interchanged. 

In conclusion the Jaquet restriction is trivial, and the $\rho$ are split in to two kinds, those that are are involution type invariant, and therefore so is the corresponding $\rho$-isotypic representation, and pairs $\rho, \rho'$ which have  $\rho^{\sigma,*}=\rho'$, their blocks are interchanged and the corresponding isotypic representations are $\sigma$ twist duals.
This finishes the proof.

%So in conclusion the Jaquet restriction is trivial, for $\rho$ of the first type, the $\rho$-isotypic representation is invariant, while for pairs $\rho,\rho'$, their isotypic parts are interchanged by the involution, finishing the proof.

\subsection{Structure of The Paper}

In \S \ref{sec:2} we go over preliminaries as well as some notions and basic claims.

In \S \ref{sec:2.1} we go over the possible involutions of $GL_{n}$.

In \S \ref{sec:2.2} we go over the definition and basic facts of Deligne-Lusztig induction, and prepare the use of \cref{luzTh}.

In \S \ref{sec:2.3} we go over the basics of Zelevinsky theory and prove \cref{zelThm}.

In \S \ref{sec:3} we prove the cuspidal case of the main theorem.

In \S \ref{sec:4} we prove the induction lemma.

In \S \ref{sec:5} we prove the main result. 

\subsection{Acknowledgments}

I would like to thank Avraham Aizenbud for suggesting this problem, helping me along the way, and teaching me mathematics. 

I was partially supported by ISF 249/17 and BSF 2018201 grants.
\section{Preliminaries} \label{sec:2}

%Maybe?===
%yesbe, I think i wasn't sure if the whole section is needed, its needed.

\subsection{Involutions of $GL_{n}$}\label{sec:2.1}

%This whole section will probably be bad at the start at least===
%Maybe redundant? maybe also talk about what is known?===
%Kind of a funny comment, I put in other places everything that is relevent.

In this section, we will discuss the possible involutions for $GL_{n}$, and set some notions we will use.

%Reductive or semi-simple?###
%defintily not true for general reductive as written, should be true for GLn but i should find a good reasoning.

As is known, the automorphism group of a semi-simple group is an extension of the automorphisms of the Dynkin diagram by the inner automorphisms. In other words, we can lift any automorphism of the Dynkin diagram to an automorphism of the group, and if we allow to conjugate these automorphisms, these are all automorphisms.

This extends to $GL_{n}$ even though it is not semi-simple.

With this in mind, $GL_{n}(\bb{F}_{q})$ (as a group over $\bb{F}_{q})$ has two automorphisms of the Dynkin diagram, corresponding to the identity and to $X\rightarrow X^{-t}$, and a general automorphism will be a conjugation of one of those.

$GL_{n}(\bb{F}_{q^{2}})$ as a group over $\bb{F}_{q}$, we also have the frobenious automorphism, so have in addition conjugations of $X\rightarrow Fr(X)$ and $X\rightarrow Fr(X)^{-t}$.

%again, maybe bad notiation===
%I decided to go with type anyway.
We will say that two automorphisms which are conjugate are of the same type. If $\sigma$ is an automorphism, we will call its type $\sigma$-type.

An important observation is that types are 'preserved' under restriction. For example if $\sigma$ is an automorphism of $GL_{n}$ which preserves $GL_{m}$ sitting inside (say, as the left upper block), then $\sigma|_{GL_{m}}$ has $\sigma$-type, even if $GL_{m}$ isn't preserved, but goes to a different block, then it would still be of $\sigma$-type (just conjugate it back to the upper left block).

Among the automorphisms, we especially care about involutions. Conjugation of $G$ yields isomorphic involutions, and conjugating by central elements does not change the involution. These two facts yield that for each type, there are a finite number of involutions:

%SAY WHAT E IS, BLOCK MATRIX ANTI-SYMETRIC BLA BLA===
%Maybe what is known here, maybe before?===
%I think before
\begin{itemize}
    \item For the $Id$-type, the conjugation matrix $A$ is either diagonal with $\pm 1$ on the diagonal, or, if the dimension is even, it can be a matrix $E$ that squares to a non-square scalar matrix. $G^{\sigma}$ in this case would be either $GL_{k}\times GL_{n-k}$ or $GL_{n/2}(\bb{F}_{q^{2}})$.

    %Sp?===
    %im not sure what my issue was?
    \item For the $(-)^{-t}$-type, the conjugating matrix would either be symmetric or antisymmetric, giving $O,Sp$ of the corresponding form.

    \item For the $Fr$-type and $Fr(-)^{-t}$-types, conjugation does not give anything else, and we have  $GL_{n}(\bb{F}_{q})\subseteq GL_{n}(\bb{F}_{q^{2}})$ for the $Fr$-type, and $U$ for the $Fr(-)^{-t}$-type.
    
\end{itemize}

%I Use tau at another place, i should change at least one of them.===
%I think it is fine, but i put # at the other place.
From now on we will use $\tau$ to denote $\sigma^{-1}$. We are concerned with proving that certain representations satisfy $\pi \cong \pi^{*,\sigma}$. Already from here we see that conjugation of $\sigma$ does not change the representation, so this condition is dependent only on the $\sigma$-type. For this reason, we will sometimes refer to this as $\sigma$-type invariant.

%NEEDED?===
%I sould keep this either way,  the wording here is a bit meh===
%I think its better now.
Since representations are determined by their characters, we can rephrase it as follows:

Let $\chi$ be the character of $\pi$, then $\pi \cong \pi^{*,\sigma}$ is the same as $\chi(g)=\chi(\sigma(g)^{-1})=\chi(\tau(g))$. For this reason, we will also call this property being $\tau$ invariant.

\begin{remark}
    Since $X,X^{t}$ are always conjugate, it is obvious that any representation is $(-)^{t}$ invariant, i.e., there is nothing to show for the $(-)^{-t}$-type. Similarly, being $Fr(-)^{-t}$-type  invariant is the same as being $Fr$ invariant.
\end{remark}

%NEEDED?===
%might as well put the remark in.
%Add refrence===
%ask guy sh?
%rami said carter maybe?

%For now i took it down, if i find a good reffrence then mayebe, but not really needed.
%\begin{remark}
%    invoultions, or equivavlinty symetric pairs have been classified in general for reductive group.
%\end{remark}

\subsection{Deligne-Lusztig}\label{sec:2.2}
%Not how it should be written.===
%Not sure whats the issue, ill give remarks
% - l not p.
% - no need for words like can and have, say things like choose, call.
% - I dont think unistropic --> cusp is true?
% - Look over the examples?
% - let us collect is not good.
%REFERENCE - I should find the delign mitchel thing and refrenced it.
% Fixed most of it.

We will only need the basics of Deligne-Lusztig induction. We will follow \cite[Chapter 7]{carter1993finite}.

Let $G$ be a group defined over $\bb{F}_{q}$, Let $T\subseteq G$ be an $\bb{F}_{q}$ defined torus, and $\theta \in T(\bb{F}_{q})^{*}$ be a character on the $\bb{F}_{q}$ points.

Choose a borel $B$ that contains $T$, which does not have to be defined over $\bb{F}_{q}$, denote by $U$ its unipotent radical, we now define:

$$ X = \{x\in G|Fr(x)x^{-1}\in U\}$$

And we will refer to this as the Lusztig variety, notice $G(\bb{F}_{q})\times T(\bb{F}_{q})$ acts on this variety by multiplication from the right and left.

\begin{definition}
    if $G,T,\theta$ are as before, we define $R_{T,\theta}$ as the virtual representation:

    $$R_{T,\theta} =  \sum_{i} (-1)^{i} H_{c}^{i}(X;\bb{Q}_{l})^{\theta}$$

    We call it the Deligne-Lusztig induction of $\theta$
\end{definition}

\begin{remark}
    There is some ambiguity in the definition as to the coefficients of $\theta$ and the choice of $l$, it is not of great importance.
    We can choose an isomorphism between $\bar{\bb{Q}_{l}}$ and $\bb{C}$ and consider these as complex characters. It can be shown that the character is independent of the choice.
\end{remark}

The following are a few facts on Deligne-Lusztig induction.

\begin{enumerate}
    \item $R_{T,\theta}$ is independent of the choice of $B$.
    \item Every irreducible representation of $GL_{n}(\bb{F}_{q})$ appears in some $R_{T,\theta}$
    \item If $\theta$ is in general position (i.e., it is not fixed by any non-trivial element of $W(T)$), then $R_{T,\theta}$ is an irreducible representation up to sign. For $ \Gl{n}{F}{q}$, if $T$ is anisotropic, then it will also be cuspidal.
    \item If $R_{T,\theta}$ and $R_{T',\theta'}$ share a common irreducible representation, then $\theta,\theta'$ are geometrically conjugate (See definition below).
\end{enumerate}

%BETTERER===
%I dont know what i wanted here.
\begin{definition}
    if $T,\theta$ and $T',\theta'$ are as in the definition of Deligne-Lusztig induction, then we say that they are geometrically conjugate if there is $n$, and $f\in G(\bb{F}_{q^{n}})$ s.t. $T'= T^{f}$ and  $\theta(N(fxf^{-1})=\theta'(N(x))$, where $N$ is the norm map $T(\bb{F}_{q^{n}})\rightarrow T(\bb{F}_{q})$ or $T'(\bb{F}_{q^{n}})\rightarrow T'(\bb{F}_{q})$.
\end{definition}

In other words, we allow to pull back a character to a field extension using the norm map, and then we allow to conjugate by an element.

\begin{example}
    In $GL_{n}(\bb{F}_{q})$ the anisotropic torus is $\bb{F}_{q^{n}}^*$, It is split over $\bb{F}_{q^{n}}$, and the norm map will be:
    $$ \bb{F}_{q^{n}}^{*,n}\rightarrow \bb{F}_{q^{n}}^{*}$$
    $$ (x_{1},....,x_{n})\rightarrow x_{1}\operatorname{Fr}(x_{2})\operatorname{Fr}^{2}(x_{3})\cdots \operatorname{Fr}^{n-1}(x_{n-1})$$

    For the split torus, the norm map is:
    $$ \bb{F}_{q^{n}}^{*,n}\rightarrow \bb{F}_{q}^{*,n}$$
    $$ (x_{1},....,x_{n})\rightarrow (N(x_{1}),...,N(x_{n-1}),N(x_{n}))$$

    And conjugating allows us to reorder the coordinates.

    We see that characters pulled back from the split torus are $Fr$ invariant characters, while ones pulled from the anisotropic torus are ones of the form $(\phi, \operatorname{Fr}(\phi),\operatorname{Fr}^{2}(\phi) ... , \operatorname{Fr}^{n-1}(\phi))$, so we see a character of the anisotropic torus is geometrically conjugate to a character of the of the split torus iff it is $Fr$ invariant, and the character on the split torus would then be a diagonal character. In both cases, we see the characters are not in general position.

    This calculation can be extended in general to see that a character of a torus of $GL_{n}$ is in general position iff it is not geometrically conjugate to any other character that it is not conjugate to.
\end{example}

We will need the following result of Lusztig from \cite{Lusztig2007}.
\begin{theorem}[{{\cite[Theorem 3.5]{Lusztig2007}}}]
    Let $\rho$ be an irreducible representation of $G(\bb{F}_{q})$ which is $H-$distinguished. Say $\rho$ appear in $H^{i}_{c}(X)^{\theta}$ (Where $X$ is the Lusztig variety corresponding to some torus $T$). then there is some $f\in G(\bb{F}_{q^{n}})$ such that $T^{f}$ is $\sigma$ invariant, and that writing $\tilde{\theta}=\theta(N(f^{-1}tf))$ a character on $T^{f}(\bb{F}_{q^n})$, then $\tilde{\theta}(\sigma(t)) = \tilde{\theta} (t)^{-1}$.
\end{theorem}

%Check notiotions, fix in the intro===
%Im tallking here about the tau-invariant notion, also some wording needs to be fixed.
%I think this is fine now.
In particular this proves the following:
\begin{theorem} 
    Let $G$ be an algebraic group over $\bb{F}_{q}$, $T,\theta$ a maximal torus and a character on the rational points, $\sigma,H$ an involution and its fixed points.  If $R_{T,\theta}$ is $H-$distinguished, $(T,\theta)$ is not geometrically conjugate to any other pair (that it is not rationally conjugate to) and $\theta$ is in general position, then $R_{T,\theta}$ is $\tau$ invariant.
\end{theorem}

%The fact R_{T,\theta} is irr might not be obvious===
%More words?
%except the irreducible thing this is mostly fine, maybe a bit more info.
%I added the condition that $\theta$ is in general position from which it is ok, (still need details) but maybe the condition follows.
\begin{proof}
    Because there are no geometrically conjugate characters, it follows that $R_{T,\theta}$ is an irreducible representation (up to sign). If it is $H-$distinguished, then by Lusztig theorem $\theta$ and $\theta^{-1}\circ \sigma$ are geometrically conjugate, but $\theta$ is only geometrically conjugate to characters it is conjugate to, i.e. $\sigma(T)=T^{f}, \theta^{-1}\circ \sigma = \theta^{f}$ where $f\in G(\bb{F}_{q})$, therefore:
    $$ \bar{R_{T,\theta}}^{\sigma}=R_{\sigma(T),\theta^{-1}\circ \sigma}=R_{T^{f},\theta^{f}}=R_{T,\theta}$$

    Where the last equality follows since conjugation by an element of the group gives an isomorphism of everything in the definition of the Deligne-Lusztig induction.
\end{proof}

%change zelinsky to zelevinsky everywhere ===
%Change the order of zelinsky and invoultions?=== changed

%Decide isotipic or primary?###
% I mostly went with primary, it isn't a very good choice as this definition allready exists, should find a better name.

\subsection{Zelevinsky Theory}\label{sec:2.3}
Here we briefly review Zelevinsky theory describing Irreducible representations in terms of the cuspidal ones. We follow \cite{zel}.

\begin{definition}
A Hopf algebra over a comutative ring $K$ is a graded $K$ module $R=\bigoplus_{n\geq0} R_{n}$ with graded morphisms (over $K$) $m:R\otimes_{K} R\rightarrow R$, $m^*:R\rightarrow R\otimes_{K} R$, $e:K\rightarrow R$, $e^*:R\rightarrow K$ s.t $m,e$ turn $R$ to a ring, $m^*,e^*$ turns it into a co-algebra, and $m^*$ is a ring map.
\end{definition} 
In other words, a Hopf algebra is both an algebra and a co-algebra, and both structures commute.
\begin{definition}
A PSH-algebra is a Hopf algebra over $\mathbb{Z}$ together with a free homogeneous basis for $R$ (as an abelian group), s.t. $R_{0}\cong \mathbb{Z}$ with isomorphisms $e,e^*$, all the maps of the Hopf algebra take basis elements to a positive-sum of basis elements, and with respect to the inner product in which our basis is an orthonormal basis, $m,m^*$ and $e,e^*$ are adjoint pairs.
\end{definition}

The main example to keep in mind is $R_{n}=R(G_{n})$ where $G_{n}$ is a sequence of groups like $S_{n}$ or $GL_{n}(\bb{F}_{q})$, product will be a form of induction, while the co-product will be a form of restriction.

\begin{definition}
    A basis element of a PSH-algebra is called irreducible, an element $x$ s.t $m^*(x)=x\otimes 1 + 1 \otimes x$ is called primitive. 
\end{definition}
It is easy to see that in the example $R=\oplus_{n} R(GL_{n}(\bb{F}_{q}))$, primitive means cuspidal, so irreducible primitive means irreducible cuspidal.

The next two theorems classify PSH-algebras:

\begin{theorem}
Let R be a PSH-algebra, denote $\mathscr{C}$ as the set of irreducible primitive elements, then $R\cong \bigotimes_{\rho \in \mathscr{C}} R(\rho)$ where $R(\rho)$ is a PSH-algebra with only one irreducible primitive element, which is the sub PSH-algebra generated by $\rho$.  
\end{theorem}

\begin{theorem}
All PSH-algebras with one irreducible primitive element of degree $1$ are isomorphic, This unique algebra has exactly two automorphisms. 
\end{theorem}
\begin{remark}
    The condition on degree is technical, If the irreducible primitive element is of degree $n$, then it is the same ring but degrees multiplied by $n$.
\end{remark}

In the proof it is shown that if $\rho$ is the irreducible primitive element, then $\rho^{2}$ is the sum of two irreducible elements $x_{2},y_{2}$, the non uniqueness of the isomorphism comes from the choice between these two, in other words the non-trivial automorphism interchanges these two.

%Better wordings?===
%Applying the classifiction to the case of $GL_{n}$ we get:
specializing the classifications theorem to the PSH-algebra coming from $GL_{n}$ we get:
\begin{theorem} \label{ZelPind}
    Every Irreducible representation of $GL_{n}(\bb{F}_{q})$ is of the form $i_{L}^{GL_{n}}(\pi)$ where $L=\prod_{i} GL_{n_{i}}$, $\pi = \otimes_{i} \pi_{i}$ and $\pi_{i}$ is a $\rho_{i}$-primary representation for distinct $\rho_{i}$.
\end{theorem}

The final thing we need from this theory is the following result:

%Invoutlion types are annoying===
%invoultions types are disscussed before.
%fix whitiker everywhere ===
%I think i did, ill add to make sure to the general stuff.
\begin{theorem}\label{ZelRes}
    If $\rho$ is a cuspidal representation, invariant under an involution type, then so is any $\rho$-isotopic representation.
\end{theorem}

%reffrence? Proof. Probably diffrent names?===
%i landed on isotipic, still need refrence 

%The original comment was bellow, this is tallking about the whitiker stuff, a reffrence is a good idea, maybe ask eitan?
%Maybe more details in why this is an automorphism.
%the latex of R(GL_n) should definitly be fixed.
%add a line about why the sigma dual is also whitikar.

\begin{proof}
First, let us notice that either type gives an automorphism of $\oplus R(GL_{n})$.

%maybe more details?===
% moved this comment above

It is clear then that since $\rho$ is mapped to itself, $R(\rho)$ is mapped to $R(\rho)$, so it is enough to show that $x_{2},y_{2}$ remain the same. The two can be distinguished by the Whittaker character, if we take any $\phi:\bb{F}_{q}\rightarrow \bb{C}^*$, we can extend it to a character of the unipotent radical of the borel, by setting $\phi(U)=\sum_{i} \phi(a_{i,i+1})$. Then, for any representation, we can ask whether there is a $\phi$-equivariant vector. 

It is classical that if $\rho$ is a cuspidal irreducible representation, then $i (\rho^k)$ has a one dimensional subspace of $\phi$-equivariant vectors, so exactly one of the irreducible representations inside of this have such a vector, call this representation the Whittaker one.

An important observation is that using the torus, which normalizes the unipotent radical, we can conjugate $\phi$ to $\phi'$ for any different choice, so the Whittaker representation is independent of the choice.

Now we only need to observe that for any involution type, if $\pi$ is Whittaker with character $\phi$, then $\pi^{*,\sigma}$ is Whittaker with a different character.

So if $y_{2}$ is the Whittaker representation, then $y_{2}^{*,\sigma}$ is also Whittaker, and since it is either $x_{2}$ or $y_{2}$, it must be itself, and we are done.

%... COMPLETE THIS?===
%I think this is a remmennnet from long ago, the proof is complte (even if not preety).

\end{proof}

%Probably something more words===
%Im not sure more words are needed.

%\begin{remark}
%In general the whitiker in $Pind \rho ^{k}$ will be $y_{k}$, This is somewhat apparnt from the proof of the uniquness of the PSH-algebra with one cuspidal, and is done in zelinski.    
%\end{remark}

\section{The Cuspidal Case}\label{sec:3}

\subsection{Interaction with Restriction of Scalars}\label{sec:3.1}
In this section we prove \cref{IR}:

%something something===
%I think i just mean that I should look over this section.
%notes from a fast reading.
% - I should give names to restriction of scalars and stuff, will make it more readable===
% - Not skip steps, say say that Fr(Tn) = T1 and becuase of that T_1 is Fr^n invariant.===
% - the formulation of the theorem is abysmal.===
% - I think I am assuming that the group is extended restricted, this does not make sense in genera (I should consdier the case of U which is not extended) i saw that i dont use ===
% - I should give the lusztig variety a name (find one that is in use probably)===
% - should search for refrence.===
%Consider wether i only do this for $GL_{n}$ or not.===
\begin{theorem}
    The set $T,\theta$ of a maximal torus and a character of its rational points, is in bijection for  $G(\bb{F}_{q^n})$ as a group over $\bb{F}_{q^n}$, and for the restriction of scalars of $G$ to $\bb{F}_{q}$ as a group over $\bb{F}_{q}$. $R_{T,\theta}$ for corresponding elements are the same up to sign, whether we regard $G$ as a group over $\bb{F}_{q^n}$ or do restriction of scalars and consider it as a group over $\bb{F}_{q}$.
\end{theorem}

\begin{proof}
    This is straightforward.

    First, let us consider this in the following way, over the algebraic closure, the restriction of $G$ is $G^{n}$, and the Frobenius isomorphism is:
    $$  (x_{1},....,x_{n})\rightarrow (Fr(x_{n}),Fr(x_{1}),...,Fr(x_{n-1}))$$
    where $Fr$ is the Frobenius of $G$ over $\bb{F}_{q}$. 

    With this mind, let us consider what a maximal torus of the restriction is, it has to be of the form $T_{1}\times T_{2}\times....\times T_{n}$, and for it to be defined over $\bb{F}_{q}$ it needs to be preserved by Frobenius, which means that $T_{i+1}=Fr(T_{i})$, and $Fr^{n}(T_{1})=T_{1}$, i.e. its just the restriction of a maximal torus of $G$ over $\bb{F}_{q^n}$.

    This shows the bijection of tori, clearly characters of their rational points are the same, as the rational points don't change.

    Now let us consider $R_{T,\theta}$.
    choosing a borel of $G$ that contains $T$, $B\times Fr(B) \times ... \times Fr^{n-1}(B)$ is a borel of the restriction of $G$ that contains the restriction of $T$. the unipotent radical is of course $ U \times Fr(U) \times ... \times Fr^{n-1}(U)$, and then the Lusztig variety in the definition of $R_{T,\theta}$ for the restriction group is:
    $$L_{\bb{F}_{q}}^{-1}(U \times Fr(U) \times ... \times Fr^{n-1}(U)) = \{(x_{1},x_{2},...,x_{n})|Fr(x_{i})x_{i+1}^{-1} \in Fr^{i}(U), Fr(x_{n})x_{1}^{-1} \in U\}$$
    Now consider the map:
    $$ L_{\bb{F}_{q}}^{-1}(U \times Fr(U) \times ... \times Fr^{n-1}(U)) \rightarrow L_{\bb{F}_{q^n}}^{-1}(Fr^{n-1}(U))$$
    $$ (x_{1},x_{2},....,x_{n})\rightarrow x_{n}$$
    indeed it is well defined since: $$Fr^{n}(x_{n})x_{n}^{-1} =Fr^{n-1}(Fr(x_{n})x_{1}^{-1}) \cdot Fr^{n-2}(Fr(x_{1})x_{2}^{-1})\cdot ... \cdot Fr(Fr(x_{n-2}) \cdot x_{n-1}^{-1})\cdot Fr(x_{n-1})x_{n}^{-1}$$
    which is indeed in $Fr^{n-1}(U)$.

    The fibers of this map are easily seen to be $U\times Fr(U)\times Fr^{2}(U)\times ... \times Fr^{n-2}(U)$ (it is in fact a bundle), For $G=GL_{n}$, $U$ is an affine space, so this map induce an isomorphism on compactly supported cohomology up index, the map also clearly commutes with the action of $G(\bb{F}_{q^{n}})\times T(\bb{F}_{q^n})$ so we get that both $R_{T,\theta}$ are indeed the same (up to sign).

\end{proof}

\subsection{Cuspidal Representations of $GL_{n}$}\label{sec:3.2}
Here we prove the following well known theorem

%Move the theorem elsewhere?###
%I think i meant to move this to the lusztig prelimenries, still an option.
\begin{theorem}\label{GlCusp}
    Every cuspidal representation of $GL_{n}(\bb{F}_{q})$ is of the form $R_{T,\theta}$ where $T$ is an anisotropic torus, and $\theta$ is a character in general position.
\end{theorem}

%Probably just refernce directly?===
%terrible wording too, should not say that i didn't find stuff, uninteresting.

This is considered classical, it is directly implied by Lusztig general classification, but it is much easier, and can be deduced from \cite{Deligne-Lustig} theorem 6.8:
\begin{theorem}[{\cite[Theorem 6.8]{Deligne-Lustig}}]
    If $T,T'\subseteq G$ are maximal tori, and $\theta, \theta'$ are characters of their $\bb{F}_{q}$ points, then:
    $$ \left <R_{T,\theta},R_{T',\theta'}\right >_{G(\bb{F}_{q})} = |\{ w\in W(T,T')|\theta'(x^{w})=\theta(x)\}|$$
    Where $W(T,T')$ is $T$-orbits that conjugate $T$ to $T'$. 
\end{theorem}
From this follows that if $\theta$ is in general position, then $R_{T,\theta}$ is an irreducible representation (up to sign), and that all of these are different (up to $G$-conjugation), in particular this means that if $T$ is anisotropic and $\theta$ is in general position then $R_{T,\theta}$ is an irreducible cuspidal representation. Now, by counting, we can see that these are all the cuspidal representations, proving the theorem.
\begin{proof}[{Proof of Theorem \cref{GlCusp}}]
We claim by induction that the number of cuspidal irreducible representations is $D_{n}$, the number of elements in $\bar{\bb{F}_{q}}^{*}$ of degree exactly $n$, divided by $n$.

For the case $n=1$, this is trivial.
Assume this is known for $x<n$, the number of irreducible representations of $GL_{n}(\bb{F}_{q})$ is the number of conjugacy classes, which can be described combinatorially as follows.

we have to choose distinct $\lambda_{1},\lambda_{2},...,\lambda_{k} \in \bar{\bb{F}_{q}}^*$ defined up to Frobenius, each of degree $d_{i}$, we have numbers $a_{1},...,a_{k}$ such that $\sum a_{i}d_{i}=n$, now for each $i$ we have to choose a partition of $a_{i}$.

notice that $D_{n}$ is a subset of this choice where $k=1, a_{1}=1$.

Now for each such $d_{i},a_{i}$ with $k>1$ or $a_{1}>1$ we have the Levi subgroup $\prod GL_{d_{i}}(\bb{F}_{q})^{a_{i}}$, by induction the number of cuspidal representations of it which are the same on $GL_{d_{i}}$ for the same $i$ and are different otherwise, is exactly the number of choices for $\lambda_{i}$, and by the theory of parabolic induction for $GL_{n}$, the parabolic induction splits into exactly the product for the number of partitions of $a_{i}$.

Therefore, the remaining representations are all the cuspidal ones, and there are exactly $D_{n}$ of them.

Since for the anisotropic torus, there are exactly $D_{n}$ characters in general position (up to conjugation), and for each one of them $R_{T,\theta}$ is a different cuspidal irreducible representation, these must be all of them, and the theorem is proved.
    
\end{proof}

Now the proof of \cref{theorem A} for cuspidal representations is a straightforward application of \cref{LT}.

%BETTER WORDS===
%Im not sure what my issue here was, i guess i can say more details, but it is a direct use of the theorems.
\begin{proof}[Proof of \cref{theorem A}]
    Let $\pi$ be an irreducible cuspidal representation of $GL_{n}(\bb{F}_{q}^{m})$ which is $H$-distinguished. 
    By $\cref{GlCusp}$ we know $\pi=\pm R_{T,\theta}$ where $T$ is anisotropic and $\theta$ is in general position, It then follows form \cref{LT} that $R_{T,\theta}$ is $\tau$ invariant, and we are done.
\end{proof}

\section{Induction Lemma}\label{sec:4}
In this section we will prove $\cref{PindLem}$.

%MAKE THE NAME HERE BETTER===
%probably the name of the section? i think its fine actually

%FIX IN THE INTRO And here?
%not sure what i meant probably to go over the formulation and make sure its the same as in the intro

%POSSIBLE REMARK THAT I DIDNT SPECIFY WHAT JAQUES RESTRICTION, BUT I DONT CARE?===
%Are the jaquet restrictions diffrent? I dont think so, I should make sure.
%If it is true, I should say that while it doesnt matter this is not the standart restriction, otherwise i should add this to the formulation.
%ADDED A REMARK

%I Had the lamma we here before, but it was decided to put it later===
%\begin{lemma}
 %   Let $G= Gl_{n}(\bb{F}^{q^{m}})$ where $m$ is $1$ or $2$, $\sigma$ an invoultion of $G$ defined over $\bb{F}_{q}$, and $H=G^{\sigma}$. let $L\subseteq P$ be a levi subgroup and parabolic of $G$ , $\pi$ a representation of $L$. If $Pind_{L}^{G}(\pi)$ is $H$-distinguished, then there is a levi subgroup $L'\subseteq L$ and an element $A\in G$ such that the map $X\rightarrow A\sigma (X)A^{-1}$ is an involution, preserves $L'$, and $jres_{L'}^{L}(\pi)$ is distinguished with respect to this involution.
%\end{lemma}

%Now that the formulation isnt at the start, I Should make sure everything is still readable.===
%Seems readable.

%also check conditions of invoultions and parabolic maybe?===

%not sure this remark is needed.
%Might as well keep it.

Throughout this section, let $P$ be a standard parabolic, $\sigma$ an involution, $H=G^{\sigma}$, $\pi$ is a representation of the Levi of $P$, and we will assume that the parabolic induction of $\pi$ to $G$ is $H$-distinguished.
\begin{remark}
    Since we work with finite groups, there is no difference between having invariant functionals (i.e., being distinguished) and invariant vectors, and so we will work with them interchangeably. 
\end{remark}

We start by using Mackey theorem to get:

\begin{claim} \label{MacCl}
    There is some $ PgH \in P\backslash G/H$ such that $\pi$ is distinguished with respect to $H^{g}\cap P$.
\end{claim}
\begin{proof}
    First, we apply mackey theory:

$$  Ind_{P}^{G}( \pi)|_{H} = \bigoplus_{g\in P\backslash G/H} Ind_{H\cap P^g}^{H} (\pi^g|_{H\cap P^g}) $$
Therefore:
$$ \InPr{Ind_{P}^{G}( \pi) |_{\rmf{H}}}{{1}_{\rmf{H}}} = \sum_{g\in P\backslash G/H} \InPr{Ind_{H\cap P^g}^{H} (\pi^g|_{H\cap P^g})}{{1}_{\rmf{H}}} = $$
$$  = \sum_{g\in P\backslash G/H} \InPr{\pi^g|_{H\cap P^g}}{{1}_{\rmf{H \cap P^g}} } = \sum_{g\in P\backslash G/H} \InPr{\pi|_{H^g\cap P}}{{1}_{\rmf{H^g \cap P}}} $$

And since the LHS is positive, then so is the RHS, so one of those terms is non-zero, which is exactly what we wished to prove.
\end{proof}

%Probably should be described at a diffrent point===
%This point is actually fine,  the wording is a bit meh - fixed?
%Is the name geometric lemma really the name? - I think it is.

Now we use the geometric lemma to get a partial description of $P\backslash G/H$:

\begin{lemma} \label{geo}
Let $G$ be a connected reductive algebraic group over $k$,$\sigma$ an involution of $G$ defined over $k$, $H=G^{\sigma}$ the fixed points of $\sigma$, and $P$ a minimal parabolic over $k$. Then $P$ contains a $\sigma$ stable maximal $k$ split torus $A$, and every double coset $H(k)gP(k)$ contains an element $x$ such that $x\cdot \sigma(x)^{-1} \in N(A(k))$. 
\end{lemma}

%CLEANINIG HERE===
% - say something about the torus sigma and so on===
% - im not sure about iota, changed to r.===
% - I see i just used v for a functional, so i can erase the remark before about this.===
% - Maybe give a few more words about the levi and which restriction.===
% - the calucaltions are a bit ugly. fixed them somewhat.===
% - should say some more about the r/2.did.=== 

% feels like im assuming the shape of sigma, probably not a big issue, but should be noted===

We can assume that the maximally split torus is the standard one, then $Q(x)=x^{-1}\sigma(x)$ is a permutation matrix up to scalars.

Now $H^{x}$ consists of elements $g$ s.t. $xgx^{-1}\in H$ i.e. are preserved by $\sigma$, opening this we get that the condition is:
$$ \sigma(xgx^{-1})=xgx^{-1}$$
$$ Q(x)\sigma(g)Q(x)^{-1} = g$$
using that $\sigma(Q(x))=Q(x)^{-1}$ we see that $\sigma_{x}:g\rightarrow Q(x)\sigma(g)Q(x)^{-1}$ is an involution, and $H^{x}$ is the elements fixed by it.

Now let us insert some notations, our Levi $L$ correspond to a partition of $[n]$, we will consider as a function $f:[n]\rightarrow [m]$, such that $P$ is standard: $$P=\{A\in GL_{n}|\forall_{i,j} f(i)>f(j): A_{i,j}=0\}$$

Up to scalars $Q(x)$ is a permutation matrix, let us call the permutation $r$, we define $T_{a,b}=\{i\in [n]|f(i)=a,f(r(i))=b\}$, notice that if $L'$ is the Levi corresponding to this partition then $\sigma_{x}$ acts on $L'$, and it is in fact the largest standard Levi for which this is true.

%remark about functional?===
%Added before, I think like it should be fine.

We will now prove that a Jaquet restriction to $L'$ is distinguished.
Let $v\in \pi^{*}$ be a functional fixed by $H^{x}\cap P$, first we show:
\begin{lemma}\label{UniLem}
    Let $M$ be a matrix with only non-zero coordinates at $i,j$ with $i\in T_{a,b},j\in T_{a,c}$ with $b<c$ for the $Id,Fr$ types, or with $c<b$ for the $(-)^{-t},Fr^{-t}$ types, then $(I+M)v=v$.
\end{lemma}

\begin{proof}
    If we denote by $V_{a,b}$ the space generated by the coordinates $T_{a,b}$, then we can think of $M$ as a map $M:V_{a,b}\rightarrow V_{a,c}$ extended by zero.

    Noticing $Fr$ doesn't change which coordinates are non-zero, we see that in the $Id,Fr$ cases, $\sigma(1+M)=1+N$ where $N$ is still a map $N:V_{a,b}\rightarrow V_{a,c}$. 

    Then since $r(T_{a,b})=T_{b,a}$, we see that $\sigma_{x}(1+M)=Q(x)\sigma(1+M)Q(x)^{-1}$ is $1+K$ with $K$ a map from $V_{b,a}$ to $V_{c,a}$.

    Now, since $b<c$ we see that $\sigma_{x}(1+M)$ is in the unipotent radical of $P$, in particular by the definition of parabolic induction $\sigma_{x}(1+M)v=v$.

    Now we notice that since we cannot have $T_{b,a}=T_{a,c}$ or $T_{c,a}=T_{a,b}$ since either of these would imply $b=c$, then we get $KM=0=MK$, in particular they commute, or equivavntly $1+M$ commutes with $\sigma_{x}(1+M)$.

    This now implies that: $$\sigma_{x}((1+M)\sigma_{x}(1+M))=\sigma_{x}(1+M)(1+M)=(1+M)\sigma_{x}(1+M)$$
    So $(1+M)\sigma_{x}(1+M)$ is in $H^{x}$, and it is also clearly in $P$, so we get:

    $$ v = (1+M)\sigma_{x}(1+M)v = (1+M)v$$

    And we are done in this case.

    For the types $(-)^{-t},Fr^{-t}$ things are similar.

    Since $(1+M)^{-1}=1-M$, we see that $\sigma(1+M)-1$ would end up taking $V_{a,c}$ to $V_{a,b}$ as transposing switches them.

    Then $\sigma_{x}(1+M)=1+N$ with $N$ taking $V_{c,a}\rightarrow V_{b,a}$, again this time if $c<b$ this is in the unipotent radical.

    The difference is that this time it can be that $T_{c,a}=T_{a,c}$ or $T_{b,a}=T_{a,b}$ (but not both at the same time), at which point $M,N$ do not commute.

    Let's consider the case $a=b$, the other is done similarly. looking at this we have $NM=0$ but $MN=K$ is a matrix taking $V_{c,a}$ to $V_{a,c}$, now $r$ does commute with both $M$ and $N$, so let us consider:

    $$ (1+M)\sigma_{x}(1+M)=1+M+N+K=(1+N)(1+M)(1+K)=\sigma_{x}(1+M)\cdot (1+M)\cdot (1+K)$$

    applying $\sigma_{x}$ to this equation we get:
    $$ \sigma_{x}(1+M) \cdot (1+M)= (1+M)\sigma_{x}(1+M)\sigma_{x}(1+K)=\sigma_{x}(1+M)(1+M)(1+K)\sigma_{x}(1+K)$$
    So $(1+K)=\sigma_{x}(1+K)^{-1}$,now using $(1+\frac{K}{2})^{2}=(1+K)$ we get can find that:

%THIS LINE IS TOO LONG===
%Fix graphics here
%I think these two are fixed, how ever this whole section is a bit, ugly
%explaing the r/2.

    $$\sigma_{x}(\sigma_{x}(1+M)(1+M) (1+\frac{K}{2}))=(1+M)\sigma_{x}(1+M)\sigma_{x}(1+\frac{K}{2})=$$
    $$=\sigma_{x}(1+M)(1+M)(1+K)\sigma_{x}(1+\frac{K}{2})=\sigma_{x}(1+M)(1+M)(1+\frac{K}{2})$$

    Where the last equality follows because $\sigma_{x}(1+K)=1-K$, and $\sigma_{x}$ restricted to matrices of the form $1+U$ where $U$ is from $T_{a,c}$ to $T_{c,a}$ is an isomorphism of a linear group, in particular it is a linear map.

    Now $1+\frac{K}{2}$ is also in the unipotent radical, so like before:
    $$v =  (1+M) \sigma_{x}(1+M) \sigma_{x}(1+\frac{K}{2})v = (1+M) \sigma_{x}(1+M)v = (1+M)v$$
    %NEEDS EXPLAINATIONS, ALso probably replace order===
    %Not sure what i meant by replace order, explainition is probably for r/2.
\end{proof}

Now we can complete the proof of \cref{PindLem}:
\begin{lemma}
    Let $G= GL_{n}(\bb{F}^{q^{m}})$ where $m$ is $1$ or $2$, $\sigma$ an involution of $G$ defined over $\bb{F}_{q}$, and $H=G^{\sigma}$. let $L\subseteq P$ be a Levi subgroup and parabolic of $G$ , $\pi$ a representation of $L$. If $Pind_{L}^{G}(\pi)$ is $H$-distinguished, then there is a levi subgroup $L'\subseteq L$ and an element $A\in G$ such that the map $X\rightarrow A\sigma (X)A^{-1}$ is an involution, preserves $L'$, and $r_{L'}^{L}(\pi)$ is distinguished with respect to this involution.
\end{lemma}
\begin{remark}
    While the Levi-decomposition used for the restriction does depend on the involution, and is not necessarily the standard one, all the Jaquet restrictions are actually isomorphic, so the statement is true regardless. In addition it will not be important to us which unipotent subgroup we use further in.
\end{remark}
%WRITE BETTER?===
%Possible refrence the lemma in this section? see how rami did.
%The wording seems mostly fine to me, maybe be a bit more explicit.
% Talked with rami, I should put the formulation here then just 'proof' without 'of'
%Did this.
\begin{proof}
    By \cref{MacCl} and \cref{geo} we know there is $x$ with $Q(x)\in N(A)$ s.t. $\pi$ is distinguished with respect to $H^{x}\cap P$. By \cref{UniLem} the invariant functional is also invariant under $1+M$ when $M$ takes $V_{a,b}$ to $V_{a,c}$ with $b<c$ or $b>c$ depending on the type. This kind of matrices generate the unipotent radical of a parabolic with respect to the partition $V_{a,b}$. (The order is either the lexicographic one or the lexicographic with the reversed order in the second variable). Call this unipotent radical $U$ and the Levi $L'$, then we see that $\pi^{U}$ is distinguished with respect to $L'\cap H^{x}=L'^{\sigma_{x}}$.

    $\pi^{U}$ is a Jaquet restriction of $\pi$, and we are done.
\end{proof}

%consider merging the two sections, and having subsections?===
%maybe even just one big proof?
%Did that.
\section{Proof of Main Theorem}\label{sec:5}
What remains of the proof is now done in two steps, the primary case, then the general case.

%NEEDS WORDS===
%NEED A LOT OF CLEANING
% - I should be more detailed in the use of the lemmas and theorems.===
% - I should talk more about types and the restrictions.===
% - I should change it probably the invariant to tau invariant, and remark before that it means type===

%Changes i've done. Made it into one section, with one big proof.
%I dont like the splitting currently.===

\begin{proof} [Proof of \cref{theorem A}] As discussed the proof is in two parts.

\textbf{The $\rho$-Primary Case}

Let $\pi$ be a distinguished $\rho$-isotopic representation,  then in particular $i_{P}^{G}\rho^{\otimes k}$ is distinguished for some $k$.

By \cref{PindLem} we have a Jaquet restriction of $\rho^{\otimes k}$ which is distinguished with respect to some $\sigma_{x}$, seeing as $\rho$ is cuspidal, any non-trivial Jaquet restriction is zero, so we see $\rho^{\otimes k}$ is distinguished as a representation of $GL_{r}(\bb{F})^k$ . 

$\sigma_{x}$ is an involution, so it acts on $GL_{r}(\bb{F})^{k}$ by acting on single components, or by interchanging two. So, picking the first component, we have either:
\begin{itemize}
    \item $\rho$ is distinguished with respect to $\sigma_{x} |_{GL_{r}}$. 
    \item $\rho \times \rho$ is distinguished with respect to $\sigma_{x}|_{GL_{r}\times GL_{r}}$.
\end{itemize}

Notice that in the first case, $\sigma_{x}|_{GL_{r}}$ has the same type as $\sigma$, while in the second case the involution has the form $(x,y)\rightarrow (f(y),f^{-1}(x))$ where $f$ isn't necessarily an involution, but still is of type $\sigma$.

In the first case, we know by the cuspidal case of  $\cref{theorem A}$ that $\rho$ is $\sigma$-type invariant.

In the second case, we get $\rho \times \rho$ is distinguished by the group of elements of the form $(f(x),x)$, the restriction to this subgroup is $\rho \otimes \rho^{f}\cong Hom(\rho,\rho^{f,*})$, this having an invariant vector means that these representations are isomorphic, and since $f$ is of $\sigma$-type, this exactly means that $\rho$ is $\sigma$-type invariant again. 

Now, in either case, by \cref{ZelRes} we get that $\pi$ is also $\sigma$-invariant, and we are done.
\bigbreak

\textbf{The General Case}

%WORDS AGAIN===
%Add refrence to the theorem? Added. (It was the zelevinsky one).
%Definitly going to be needing a clean
% - say 'there is L' ' and all that jazz.===
% - I think I dont need connecting words like 'now' and the like. didn't notice, later===
% - I Think i should give names to all the partitions and rep' this is currently too many words.===
% - invariant type should be changed to tau,  the interchanges should be said with more words and maybe remarked beforehand?===
%fix jackuet spelling? also check the correct spelling - I did it, best to check again.===

Let $\pi$ be an irreducible distinguished representation. By \cref{ZelPind}, $\pi$ is the parabolic induction of a tensor of $\rho_{i}$-isotopic representations for different $\rho_{i}$, let us write the block for $\rho_{i}$ as $T_{i}$.

Using \cref{PindLem} we know there is $x$ with $Q(x)\in N(A)$ and $L'$  maximal such that $\sigma_x$ acts on it, such that the Jaquet restriction to $L'$ is $\sigma_{x}$ distinguished.

Now $L'$ is made of blocks, each one sits inside one of the blocks $T_{i}$, and so corresponds to some $\rho_{i}$. Denote these blocks $S_{i,j}$.

If we have a block $S_{i,j}$ that is preserved by $\sigma_{x}$, the Jaquet restriction to $L'$, has in the component of this block a $\rho_{i}$-isotopic representation since it is in the Jaquet restriction of $\rho_i$, and this component is distinguished with respect to $\sigma_{x}|_{S_{i,j}}$, and so by the previous part of the proof, $\rho_{i}$ is $\sigma$ type invariant.
If there are two blocks  $S_{i,j},S_{i,j'}$ that are interchanged, then we have as before two $\rho_{i}$-isotopic representations that are $\sigma$ type interchanged, in particular $\rho_{i}$ is $\sigma$-type invariant again.

finally if we have two blocks $S_{i,j}$ and  $S_{i',j'}$ that are interchanged, then there is $\rho_{i}$-isotopic representation that is $\sigma$-type interchanged with a $\rho_{j}$-isotopic representation, so in particular $\rho_{i}$ and $\rho_{j}$ are $\sigma$ interchanged.
%type thing needs proper language.===
%I think this is fine, maybe make sure that it is written that sigma-type invariant is tau invariant.

In conclusion, since each $\rho_{i}$ can only be interchanged with one other $\rho_{j}$, we get that the $\rho_{i}$ have two options.

\begin{itemize}
    \item $\rho_{i}$ is $\sigma$-type invariant, $Q(x)$ preserve $T_{i}$, in particular the $\rho_{i}$-isotopic representation is also $\sigma$-type invariant.
    \item $\rho_{i}$ is $\sigma$-type interchanged with $\rho_{j}$, $Q(x)$ interchanges $T_{i}$ and $T_{j}$ in particular the Jaquet restriction is trivial, and the $\rho_{i}$-isotopic representation is $\sigma$-type interchanged with the $\rho_{j}$-isotopic representation.
\end{itemize}

In conclusion, the $\rho_{i}$-isotopic representations are either $\sigma$-type invariant, or come in pairs of $\sigma$-type pairs, since the $\sigma$-type involution commutes with parabolic induction, and parabolic induction does not care of the order of the blocks, we conclude the $\pi$ is $\sigma$ invariant, and we are done.

\end{proof}

%For some reason the bib contains a bunch of wierd stuff, i should fix that===
%\printbibliography

\bibliographystyle{alpha}
\bibliography{main}

\end{document}